\documentclass[12pt, letterpaper,reqno]{amsart}

\usepackage[margin=1in]{geometry}
\usepackage{amsmath,amssymb,amsthm,color,url}

\title{Locating resonances on hyperbolic cones}
\author{Dean Baskin}
\email{dbaskin@math.tamu.edu}
\address{Department of Mathematics, Texas A\&M University \\ Mailstop
  3368 \\ College Station, TX 77843}
\author[J.L. Marzuola]{Jeremy L. Marzuola}
\email{marzuola@math.unc.edu}
\address{Department of Mathematics, UNC-Chapel Hill \\ CB\#3250
  Phillips Hall \\ Chapel Hill, NC 27599}

\newcommand{\pd}[1][]{\partial_{#1}}

\newcommand{\lap}{\Delta}
\newcommand{\NN}{\mathbb{N}}
\newcommand{\ZZ}{\mathbb{Z}}
\newcommand{\HH}{\mathbb{H}}

\newcommand{\hcone}{\ensuremath{HC(\mathbb{S}^1_\rho)}}
\newcommand{\g}{\ensuremath{\mathrm{g}}}
\newcommand{\upd}{\ensuremath{\mathrm{d}}}

\newcommand{\reals}{\mathbb{R}}
\renewcommand{\Im}{\operatorname{Im}}

\newtheorem{theorem}{Theorem}
\newtheorem{prop}[theorem]{Proposition}
\newtheorem{cor}[theorem]{Corollary}
\newtheorem{remark}[theorem]{Remark}

\begin{document}

\begin{abstract}
  In this note we explicitly compute the resonances on hyperbolic
  cones.  These are manifolds diffeomorphic to $\reals_{+}\times Y$
  and equipped with the singular Riemannian metric
  $dr^{2} + \sinh^{2}r \ h$, where $Y$ is a compact manifold without
  boundary and $h$ is a Riemannian metric on $Y$.  The calculation is
  based on separation of variables and Kummer's connection formulae
  for hypergeometric functions.  To our knowledge this is the one of
  the few explicit resonance calculations that does not rely on the
  resolvent being a two-point function.
\end{abstract}

\maketitle

\section{Introduction}
\label{sec:introduction}

In this note we explicitly calculate all of the resonances for a
hyperbolic cone in terms of the eigenvalues of the cross-section.  To
fix notation, let $X$ be a manifold of dimension $n+1$ diffeomorphic
to $(\reals_{+})_{r}\times Y$, where $Y$ is a compact $n$-manifold
without boundary.  Given a Riemannian metric $h$ on $Y$, we equip $X$
with the hyperbolic conic metric $dr^{2} + \sinh^{2}r\,h$.  Except in
the special case of hyperbolic space, $X$ has an isolated conic
singularity at $r=0$.

Given a hyperbolic cone $X$ and its associated metric $g$, we define
the resolvent 
\[
R(\lambda) = \left(-\lap _{g} - \lambda^{2} - \frac{n^{2}}{4}\right)^{-1},
\]
which is a bounded operator $L^{2}(X,g) \to L^{2}(X,g)$ for
$\Im \lambda > 0$ (with the possible exception of finitely many poles).
The resolvent $R(\lambda)$ admits a meromorphic continuation from
$\{ \Im \lambda > 0\}$ to the complex plane as an operator
$L^{2}_{c}(X,g) \to L^{2}_{\operatorname{loc}}(X,g)$, i.e., from
compactly supported functions to locally $L^{2}$ functions.  The poles
of this meromorphic continuation (aside from potentially finitely many
eigenvalues lying in the upper half plane) are called
\emph{resonances}.

This note establishes the following theorem:

\begin{theorem}
  \label{thm:main-thm}
  Let $\{\mu_j^2\}_{j\in \NN}$ be the eigenvalues of $-\Delta_h$.  The
  resonances of $-\lap_{g}$ are given by
  \[
  \lambda_{j,k} = -\imath \left(  \frac12 + k +  \sqrt{\left( \frac{n-1}{2}\right)^{2}
      + \mu_{j}^{2}}  \right)
  \]
  for $k \in \NN = \{ 0, 1, 2, \dots\}$, and $j$ so that
  \[
  \sqrt{\left( \frac{n-1}{2}\right)^{2}
    + \mu_{j}^{2} } \notin \frac12  + \ZZ.
  \]
  Here an eigenvalue $\mu_{j}^{2}$ with multiplicity $m$ for
  $-\lap_{h}$ adds multiplicity $m$ to $\lambda_{j,k}$.
\end{theorem}

Note that if
\[
\sqrt{\left( \frac{n-1}{2}\right)^{2}
  + \mu_{j}^{2}} \in \frac12  + \ZZ,
\]
then $\mu_j$ contributes no resonances to $-\Delta_g$. 

In fact, in the proof of Theorem~\ref{thm:main-thm} we find \emph{all}
poles of the meromorphic continuation of the resolvent.  A simple
calculation shows that $\lambda_{j,k}^{2} + \frac{n^{2}}{4} \leq 0$
for all $j$ and $k$, so that none of the poles corresponds to an
eigenvalue.  Although the hyperbolic cones we consider do not fit into
the framework of Patterson--Sullivan theory~\cite{Patterson-fuchsian, Sullivan},
the lack of any eigenvalues is in line with their characterization of
the bottom of the spectrum.  Indeed, for convex co-compact quotients
of hyperbolic space, the Laplacian has an eigenvalue in $(0, n^{2}/4)$
only when the dimension of the limit set (and hence the trapped set)
is large enough.  The hyperbolic cones considered here have no
trapping, so the lack of eigenvalues in this range should not be surprising.

By appealing to the standard Weyl law on compact manifolds,
Theorem~\ref{thm:main-thm} admits the following immediate consequence:
\begin{cor}
  \label{cor:weyl}
  Suppose that $(Y,h)$ is generic in the sense that
  \begin{equation*}
    \sqrt{\left( \frac{n-1}{2}\right)^{2} + \mu_{j}^{2}} \notin
    \frac{1}{2} + \ZZ
  \end{equation*}
  for all $j$.  Then the resonances on the hyperbolic cone $(X,g)$
  obey the following Weyl law:
  \begin{equation*}
    \# \left\{ \lambda_{j,k} : |\lambda_{j,k}| \leq \lambda\right\} =
      \frac{|B_{n}|}{(2\pi)^{n}(n+1)}\operatorname{Vol}(Y,h)
      \lambda^{n+1} + O(\lambda^{n}).
  \end{equation*}
\end{cor}

\begin{remark}
  Many examples of generic $(Y,h)$ exist.  Although the standard
  metrics on the sphere and flat torus are non-generic, there are
  arbitrarily small generic perturbations of them.  Indeed, if $h_{0}$
  is the round metric on the sphere and $\alpha$ is any transcendental
  number, then $\alpha^{2}h_{0}$ satisfies the genericity condition.
  In particular, this shows that resonances can exist in principle in
  any dimension, even or odd.
\end{remark}

A necessary ingredient for Theorem~\ref{thm:main-thm} is of course the
meromorphic continuation of the resolvent in this setting.  We provide
only a sketch of that argument as it is nearly identical to the proof
provided by Guillarmou--Mazzeo~\cite[Section
3.3]{guillarmou2012resolvent}, which is in turn based on the
parametrix of Guillop{\'e}--Zworski~\cite{guillope-zworski1,
  guillope-zworski2} and previous analysis of
Sj{\"o}strand--Zworski~\cite{sjostrand1991complex}.  One defines the
parametrix
$Q(\lambda) = \tilde{\chi}_{0}R_{0}(\lambda)\chi_{0} +
\tilde{\chi}_{\infty}R_{\infty}(\lambda) \chi_{\infty}$,
where $R_{i}(\lambda)$ are model resolvents ($R_{0}$ is the resolvent
on a compact manifold with a conic singularity containing a large
compact region of $X$ and $R_{\infty}$ is the resolvent on a smooth
manifold hyperbolic near infinity) and $\chi_{i}$ and
$\tilde{\chi}_{i}$ are appropriately chosen cutoff functions.
Applying $(-\lap - \lambda^{2} - n^{2}/4)$ yields a remainder of the
form $I + \sum [-\lap_{X}, \tilde{\chi}_{i}]R_{i}(\lambda) \chi_{i}$.
Because the inclusion of the Friedrichs domain into $L^{2}$ is compact
on the compact piece,\footnote{For an explicit characterization of
  the Friedrichs domain, we direct the reader to
  Melrose--Wunsch~\cite[Proposition~3.1]{MW}.} we can use
the well-known mapping properties of the hyperbolic resolvent to
conclude that the remainder is a meromorphic Fredholm operator that is
invertible for large $\Im \lambda$.  Applying $R(\lambda)$ to both
sides and inverting the remainder shows that $R(\lambda)$ has a
meromorphic continuation.

One can also compute the poles of the \emph{scattering matrix} in this
setting.  Indeed, the scattering matrix can in general be written in
terms of an extension operator and the
resolvent~\cite{guillope-zworski1,borthwick2002scattering,
  graham2003scattering,guillarmouMRL, dzbook}.  With possibly
countably many exceptions, the poles of the scattering matrix are
precisely those of the resolvent.  These other poles (as in the case
of odd-dimensional hyperbolic space, which has no resolvent poles but
infinitely many scattering poles) typically arise as poles of the
extension operator and are localized on the boundary of the
conformal compactification.  Poles of the scattering
matrix not arising as poles of the resolvent provide a deep connection
between the scattering matrix and the conformal geometry of the
boundary at infinity~\cite{guillarmouMRL}.

Although many explicit calculations of resonances exist in the setting
of potential scattering on Euclidean space, many fewer such
calculations are known in geometric scattering theory.  To our
knowledge, Theorem~\ref{thm:main-thm} is one of the few explicit
calculations of resonances where the resolvent is not necessarily a
two-point function (i.e., where the resolvent depends on more than the
distance between two points).  The main such setting in which exact
calculations are known is that of quotients of hyperbolic space.  For
hyperbolic cylinders, the exact resonance structure was worked out by
Epstein~\cite{epstein:unpublished} and
Guillop{\'e}~\cite{guillope:1990}.  For hyperbolic surfaces,
Borthwick--Philipp~\cite{borthwick2014resonance} worked out the
resonances for hyperbolic warped products and
Datchev--Kang--Kessler~\cite{DKK} studied resonances of surfaces of
revolution obtained by removing a disk from a cone and attaching a
hyperbolic cusp.  Further examples for hyperbolic surfaces can be
found in Appendix~B of a paper of
Patterson--Perry~\cite{patterson-perry} and in the book by
Borthwick~\cite{borthwick2007spectral}.

In other settings, it is sometimes possible to work out the asymptotic
distribution of the resonances.  For example, S{\'a}
Barreto--Zworski~\cite{sabarreto-zworski} worked out resonances on the
Schwarzschild black hole asymptotically lie on a lattice, while
Stefanov~\cite{stefanov2006sharp} described the asymptotic
distribution of resonances exterior to a disk in Euclidean space.

On perturbations of $\reals^{3}$, solutions of the wave equation (or
the wave equation with a potential) have a resonance expansion on
compact sets.  The resonances of $-\lap$ (or $-\lap + V$) provide the
rates of decay and modes of oscillation seen in this expansion.  On
hyperbolic spaces, solutions of the corresponding wave equation also
have a resonance expansion (see, for example, the work of
Datchev~\cite{Datchev:2016} and the references therein).  Recent work
of the first author and collaborators~\cite{BVW1,BVW2} shows that
resonances on some asymptotically hyperbolic spaces also provide the
decay rates for solutions of the wave equation on asymptotically
Minkowski spaces.  One interpretation of the difference in decay rates
for the wave equation in even- and odd-dimensions is that
even-dimensional hyperbolic space has resonances, while
odd-dimensional hyperbolic space does not.

The proof of Theorem~\ref{thm:main-thm} has two main steps.  First, we
use the warped product structure to construct an explicit
representation of the resolvent on hyperbolic cones
(Proposition~\ref{prop:resolvent-formula}).  This formula is found by
using a coordinate representation (that is essentially the same as the
one used in Patterson's computation of the hyperbolic
resolvent~\cite{patterson}) and then applying Kummer's connection
formula for hypergeometric functions.  We then compute poles of the
resolvent by analyzing our resolvent formula; the poles arise as poles
of the relevant Gamma functions.  To deal with the other values of the
parameters, we rely heavily on the formulae for hypergeometric
functions coming from \cite{NIST:DLMF,Olver:2010:NHMF}.  In the final
section of the paper, we give some explicit resonance expansions for a
handful of relevant model problems.

{\sc Acknowledgments.} {The first author was supported in part by
  U.S. NSF Grant DMS--1500646.  The second author was supported in
  part by U.S. NSF Grants DMS--1312874 and DMS-1352353. We wish to
  thank Jesse Gell-Redman, Luc Hillairet, Laura Matusevich, Rafe Mazzeo and Jared
  Wunsch for helpful conversations during the
  production of this work.  We also thank the anonymous referee for giving us
  many useful suggestions for dramatically improving the exposition for the result, especially relating 
  to carefully explaining the meromorphic continuation of the resolvent, constructing the resolvent for this family of manifolds and
  the connection between the resolvent and the scattering matrix.}


\section{An Explicit Expression For The Resolvent}

We consider the resolvent $R(\lambda)$ given by
\begin{equation}
  \label{eq:resolvent-def}
  R(\lambda) = \left( - \lap_{g} - \lambda^{2} - \frac{n^{2}}{4}\right)^{-1},
\end{equation}
where we take the Friedrichs extension of the Laplacian.  We adopt the
convention that $R(\lambda)$ is a bounded operator on $L^{2}(X)$ when
$\Im \lambda \gg 0$.

Because $X$ is a warped product, we may analyze the resolvent by
separating variables.  Suppose that $\{ \phi_{j}\}$ is an
orthonormal basis of eigenfunctions of $\lap_{h}$ with eigenvalues
$-\mu_{j}^{2}$.  For a function $f$ on $X$, we then write
\begin{equation*}
  f(r,y) = \sum_{j=0}^{\infty}f_{j}(r) \phi_{j}(y).
\end{equation*}
As the $\phi_{j}$ are orthogonal, the resolvent decomposes into a
family of one-dimensional resolvents:
\begin{equation*}
  u = R(\lambda) f = \sum_{j=0}^{\infty} \left( R_{j}(\lambda)
    f_{j}\right)(r) \phi_{j}(y).
\end{equation*}
In terms of the coordinate $r$ and the eigenvalue $-\mu_{j}^{2}$,
$R_{j}(\lambda)$ has the following expression acting on
$L^{2}(\reals_{+}, \sinh^{n}r\,dr)$:
\begin{equation*}
  R_{j}(\lambda) = \left( - \pd[r]^{2} - n \coth r\, \pd[r] +
    \frac{\mu_{j}^{2}}{\sinh^{2}r} - \lambda^{2} - \frac{n^{2}}{4}\right)^{-1}.
\end{equation*}

The main result of this section is the following formula:
\begin{prop}
  \label{prop:resolvent-formula}
  For $f_{j} \in C^{\infty}_{c}(\reals_{+}\times Y)$ and $\sigma =
  (\cosh^{2} (r/2) )^{-1}$, the resolvent is given by
  \begin{align*}
&    (R_{j}(\lambda)f_{j})(\sigma) =
          \frac{\Gamma(a)\Gamma(b)}{\Gamma(c)\Gamma(1+s)} \\
    &  \hspace{.25cm}  \times \bigg[ -\int_{1}^{\sigma} f_{j}(\rho)
    F(a,b,c;\sigma)F(a,b,1+s; 1-\rho)
    \left(\frac{\sigma}{\rho}\right)^{\frac{n}{2}-\imath\lambda}\left(\frac{1-\sigma}{1-\rho}\right)^{-\frac{n-1}{4}+\frac{1}{2}s} 
    \rho^{c-2}(1-\rho)^{s}\, d\rho \\
    &\hspace{.5cm} + \int_{0}^{\sigma}f_{j}(\rho) F(a,b,c;\rho)F(a,b,1+s;1-\sigma)    \left(\frac{\sigma}{\rho}\right)^{\frac{n}{2}-\imath\lambda}\left(\frac{1-\sigma}{1-\rho}\right)^{-\frac{n-1}{4}+\frac{1}{2}s} 
    \rho^{c-2}(1-\rho)^{s}\, d\rho\bigg],
  \end{align*}
  where $a$, $b$, $c$, and $s$ are given by
  \begin{align*}
    a &= \frac{1}{2}-\imath\lambda, & 
    b &= \frac{1}{2}-\imath\lambda + s, \\ 
    c &= 1 - 2\imath\lambda, & 
    s &= \sqrt{\left( \frac{n-1}{2}\right)^{2} + \mu_{j}^{2}}.
  \end{align*}

\end{prop}

\begin{proof}
  Given some $g \in C^{\infty}_{c}(\reals_{+})$, we wish to find
  $u = R_{j}(\lambda)g$, depending meromorphically on
  $\lambda$ so that $u_{j} \in L^{2}(\reals_{+}, \sinh^{n}r\,dr)$ for
  $\Im \lambda \gg 0$ and 
  \begin{equation}
    \label{eq:radial-r-coord}
    \left( - \pd[r]^{2} - n \coth r \, \pd[r] +
      \frac{\mu_{j}^{2}}{\sinh^{2}r} - \lambda^{2} -
      \frac{n^{2}}{4}\right) u = g.
  \end{equation}

  We start by reducing to a hypergeometric equation using the variable
  $\sigma = \left( \cosh^{2} (r/2) \right)^{-1}$.  Under this change, $r\to
  0$ corresponds to $\sigma \uparrow 1$, while $r\to \infty$
  corresponds to $\sigma \downarrow 0$.  In terms of $\sigma$,
  equation~\eqref{eq:radial-r-coord} becomes
  \begin{equation}
    \label{eq:radial-sigma-coord}
    \left( - \sigma^{2}(1-\sigma)\pd[\sigma]^{2} -  \left[ (1-n) -
        \frac{3-n}{2}\sigma\right]\sigma\pd[\sigma] +
      \frac{\sigma^2 \mu_{j}^{2}}{4(1-\sigma)} - \lambda^{2} -
      \frac{n^{2}}{4}\right)u = g.
  \end{equation}
  Both $0$ and $1$ are regular singular points for this equation,
  which has indicial roots given by
  \begin{align*}
    &\alpha_{\pm} = \frac{n}{2} \pm \imath \lambda & \quad &\text{at }\sigma
                                               = 0, \\
    &\beta_{\pm} = -\frac{n-1}{4} \pm \frac{1}{2}\sqrt{\left(
    \frac{n-1}{2}\right)^{2} + \mu_{j}^{2}} &\quad &\text{at }\sigma =1.
  \end{align*}
  The requirement that $u\in L^{2}(\reals_{+}, \sinh^{n}r\,dr)$
  for $\Im \lambda \gg 0$ corresponds to requiring that
  \begin{equation*}
    u(\sigma) \sim \sigma^{\frac{n}{2}-\imath\lambda}
  \end{equation*}
  as $\sigma \downarrow 0$ and that
  \begin{equation*}
    u(\sigma) \sim \sigma^{-\frac{n-1}{4}+\frac{1}{2}\sqrt{\left(\frac{n-1}{2}\right)^{2}+\mu_{j}^{2}}}
  \end{equation*}
  as $\sigma \uparrow 1$.

  We now factor out the desired indicial behavior from $u$ and
  define the function $v$ by
  \begin{equation*}
    u = \sigma^{\alpha}(1-\sigma)^{\beta}v,
  \end{equation*}
  where $\alpha = \frac{n}{2} - \imath \lambda$ and $\beta=
  -\frac{n-1}{4}+\frac{1}{2}\sqrt{\left(\frac{n-1}{2}\right)^{2}+\mu_{j}^{2}}$
  are the preferred indicial roots above.
  Making this substitution yields the following equation for $v$:
  \begin{align*}
    g &= - \sigma^{\alpha+2}(1-\sigma)^{\beta+1}v'' -
            \sigma^{\alpha+1}(1-\sigma)^{\beta}\left( (1-n+2\alpha) -
            \sigma\left[2\alpha + 2\beta -
            \frac{n-3}{2}\right]\right)v' \\
    &\quad\quad - \sigma^{\alpha}(1-\sigma)^{\beta}\left( \alpha (\alpha -
      1) + (1-n)\alpha + \lambda^{2} + \frac{n^{2}}{4}\right) v \\
    &\quad\quad -\sigma^{\alpha+1}(1-\sigma)^{\beta-1} \left(
      \beta(\beta-1) + \frac{n+1}{2}\beta -
      \frac{\mu_{j}^{2}}{4}\right) \\
    &\quad\quad + \sigma^{\alpha+1}(1-\sigma)^{\beta} \left(
      2\alpha\beta + \alpha(\alpha-1) + \beta(\beta-1) -
      \frac{n-3}{2}(\alpha + \beta)- \frac{\mu_{j}^{2}}{4}\right)v.
  \end{align*}
  The exponents $\alpha$ and $\beta$ were chosen precisely so that the
  new equation would have $0$ as an indicial root at both $0$ and $1$
  (i.e., so that the middle two terms would vanish).  In other words,
  after dividing by $-\sigma^{\alpha+1}(1-\sigma)^{\beta}$, $v$
  must satisfy
  \begin{align*}
    -\frac{g}{\sigma^{\alpha+1}(1-\sigma)^{\beta}} &=
           \sigma(1-\sigma)v'' + \left( [1-n+2\alpha] -
             \sigma\left[2\alpha+2\beta -
                                                         \frac{n-3}{2}\right]\right)v'
    \\
    &\quad\quad - \left( 2\alpha\beta + \alpha(\alpha-1) +
      \beta(\beta-1) - \frac{n-3}{2}(\alpha+\beta) - \frac{\mu_{j}^{2}}{4}\right)v.
  \end{align*}
  Plugging in the values of $\alpha$ and $\beta$ yields the following
  equation for $v$:
  \begin{align}
    \label{eq:hypergeom}
    -g\sigma^{-\alpha-1}(1-\sigma)^{-\beta} &=
        \sigma(1-\sigma) v'' + \left( (1-2\imath\lambda) 
        -\sigma \left[ 2 - 2\imath\lambda + \sqrt{\left(
                                                         \frac{n-1}{2}\right)^{2} +
                                                         \mu_{j}^{2}}\right]\right)v' \notag
    \\
    &\quad\quad- \left( \frac{1}{2} - \imath\lambda\right) \left( \frac{1}{2} -
    \imath\lambda + \sqrt{\left(\frac{n-1}{2}\right)^{2} + \mu_{j}^{2}}\right)v.
  \end{align}
  
  Equation~\eqref{eq:hypergeom} is an inhomogeneous hypergeometric
  equation with parameters $a$, $b$, and $c$ given by
  \begin{align*}
    c &= 1 - 2\imath\lambda, \\
    a+b &= 1 - 2\imath \lambda + \sqrt{\left(\frac{n-1}{2}\right)^{2}
          + \mu_{j}^{2}}, \\
    ab &= \left( \frac{1}{2} - \imath\lambda\right) \left( \frac{1}{2} -
    \imath\lambda + \sqrt{\left(\frac{n-1}{2}\right)^{2} + \mu_{j}^{2}}\right).
  \end{align*}
  In particular, we have that
  \begin{align*}
    &a = \frac{1}{2} - \imath\lambda, \quad\quad     b = a + s, \\
    &c = 2a, \quad\quad
    s= \sqrt{\left(\frac{n-1}{2}\right)^{2} + \mu_{j}^{2}}.
  \end{align*}
  
  For generic $a$, $b$, and $c$, the solution $u_{1}$ of the
  homogeneous equation (i.e., equation~\eqref{eq:hypergeom} with
  $g=0$) that is regular at $\sigma = 0$ is given by
  \begin{equation*}
    u_{1}(\sigma) = F(a,b,c;\sigma),
  \end{equation*}
  while the solution $u_{2}$ that is regular at $\sigma = 1$ is given
  by
  \begin{equation*}
    u_{2} (\sigma) = F(a,b,a+b+1-c;1-\sigma),
  \end{equation*}
  where $F(a,b,c;z)$ denotes the standard hypergeometric function with
  parameters $a$, $b$, and $c$.  In general, these two solutions are
  linearly independent and one can compute their Wronskian using
  standard facts about hypergeometric functions and Kummer's
  connection formula~\cite[15.10.17]{NIST:DLMF}:
  \begin{equation*}
    W[u_{1},u_{2}](\sigma) =
    \frac{\Gamma(c-1)\Gamma(a+b-c+1)}{\Gamma(a)\Gamma(b)}(1-c)\sigma^{-c}(1-\sigma)^{c-a-b-1}. 
  \end{equation*}
  
  This yields the following formula for the solution of
  equation~\eqref{eq:hypergeom} that is regular at both $0$ and $1$:
  \begin{align*}
    v(\sigma) = \frac{\Gamma(\frac{1}{2}-\imath
                    \lambda)\Gamma(\frac{1}{2}-\imath\lambda + s)}
                    {\Gamma(1-2\imath \lambda)\Gamma (1 + s)} 
                     &\bigg[ -u_{1}(\sigma)\int_{1}^{\sigma}
      g(\rho)u_{2}(\rho)\rho^{c-\alpha-2}(1-\rho)^{1+a+b-c-\beta-1} \,d\rho\\
    &\quad\quad\quad\quad + u_{2}(\sigma)
      \int_{0}^{\sigma}g(\rho)u_{1}(\rho)\rho^{c-\alpha-2}(1-\rho)^{1+a+b-c-\beta-1}\,
      d\rho\bigg].
  \end{align*}
  Here the exponents are given by
  \begin{align*}
    c - \alpha - 2 &= -1 - \frac{n}{2} - \imath\lambda, \\
    1 + a + b - c-\beta-1 &= \frac{n-1}{4} + \frac{1}{2} \sqrt{\left(
                            \frac{n-1}{2}\right)^{2} + \mu_{j}^{2}}.
  \end{align*}
  
  Multiplying $v$ by $\sigma^{\alpha}(1-\sigma)^{\beta}$ finishes
  the proof.
\end{proof}

\subsection{Connection to Past Resolvent Computations}

The resolvent formula in Proposition~\ref{prop:resolvent-formula}
allows us to quickly recover a formula for the resolvent on hyperbolic
space.  The resolvent on $\mathbb{H}^{n+1}$ is a two-point function
and so its integral kernel $K_{\lambda}(z,z')$ depends only on the
distance between $z$ and $z'$.  It thus suffices to consider the case
where the pole is at the origin (and so only the first integral is
relevant).  As the delta function is spherically symmetric, only the
zero mode on the cross-section ($\mathbb{S}^{n}$) contributes to the
formula and so we should multiply our formula by $1/(2\pi)$ to account
for the eigenfunction.  Accounting for the natural scaling for the
delta function on hyperbolic space then yields an expression matching
those in the literature (see, e.g., the work of
Guillarmou--Mazzeo~\cite[Section 3.1]{guillarmou2012resolvent}).

\section{Locating the resonances}
\label{sec:locat-reson-stat}

We handle resonances on a case by case basis using the structure of
the hypergeometric and Gamma functions in the expression
for the resolvent.  Since the Gamma functions will play a dominant
role in our discussion, we note that Gamma functions have simple poles at
the non-positive integers, which we denote here by $\ZZ_-$.  This
motivates the following case by case analysis.  The main observation
(noted in standard special functions texts~\cite[15.2.2]{NIST:DLMF}) is that
the function
\begin{equation*}
  \frac{1}{\Gamma (c)}F(a,b,c;z)
\end{equation*}
is entire in $a$, $b$, and $c$.  

\subsection{$c \notin \ZZ_-$}

When $c \notin \ZZ_-$, the hypergeometric functions occurring in the
resolvent formula in Proposition \ref{prop:resolvent-formula} are
regular.  Hence, the only possible poles occur due to the $\Gamma$
prefactors.  Since it is impossible for $a$ to be a
negative integer since $c = 2a$, we have two remaining scenarios.

\begin{enumerate}

\item {$c \notin \ZZ_-$,  $b \notin \ZZ_-$,  $a \notin \ZZ_-$:}

{\bf No resolvent poles}:  In this case, the $\Gamma$ function pre-factor for the resolvent has no poles, and since none arise from the hypergeometric functions (see Formula $(15.2.1)$ from \cite{NIST:DLMF}), there are no poles.

\item {$c \notin \ZZ_-$,  $b \in \ZZ_-$, $a \notin \ZZ_-$:}

{\bf Resolvent poles}:  In this case $b$ is a pole of the
Gamma function.  We see that this pole is non-removable as we can
easily see that the hypergeometric functions in the resolvent formula
are non-zero.

\end{enumerate}

\subsection{$c \in \ZZ_-$}

When $c \in \ZZ_-$, we must work a bit harder and examine the
interplay between the parameters $a$, $b$, and $c$.  Because the
function $\frac{1}{\Gamma(c)}F(a,b,c;z)$ is entire in $a$, $b$, and
$c$, the poles must arise as poles of $\Gamma (a) \Gamma (b)$.   This
leads naturally to four scenarios.

\begin{enumerate}

\item {\bf $c \in \ZZ_-$,  $b \notin \ZZ_-$,  $a \notin \ZZ_-$:}

  {\bf No resolvent poles}: The numerator $\Gamma(a) \Gamma(b)$ does
  not have poles here, so there is no resonance in this case.

\item {\bf $c \in \ZZ_-$,  $b \notin \ZZ_-$,  $a \in \ZZ_-$:}

{\bf No resolvent poles}:  The apparent pole arising from $\Gamma(a)$
is in fact removable.  Indeed, the power series expansion~\cite[15.2.1]{NIST:DLMF}
\begin{equation*}
  \frac{\Gamma (a)\Gamma(b)}{\Gamma (c)}F(a,b,c;z) =
  \sum_{k=0}^{\infty}\frac{\Gamma(a+k)\Gamma(b+k)}{\Gamma(c+k)k!}z^{k} 
\end{equation*}
and the fact that $c = 2a \leq a$ show that the resolvent is regular here.

\item {\bf $c \in \ZZ_-$,  $b \in \ZZ_-$, $a \notin \ZZ_-$:}

  {\bf No resolvent poles}: As in the previous case, the resolvent has
  a removable singularity; the same power series expansion (and the
  fact that $c \leq b$) shows that the pole is removable.

\item {\bf $c \in \ZZ_-$,  $b \in \ZZ_-$, $a \in \ZZ_-$:}

  {\bf Resolvent poles}: In this case, we again see that
  $(\Gamma(a)/\Gamma(c)) F(a,b,c;z)$ has a removable singularity at
  $a$.  However, as can be seen once again from Formula $(15.2.1)$
  of~\cite{NIST:DLMF}, the pole of $\Gamma(b)$ is a pole of the
  resolvent.  That it is a true pole (and not a removable singularity)
  follows from the observation that the power series of 
  $$\frac{\Gamma (a) }{
  \Gamma(c) } F(a,b,c;z)$$ 
  is non-zero.
\end{enumerate}

Having understood all cases, we can now finish the proof of the main theorem.

\begin{proof}[Proof of Theorem~\ref{thm:main-thm}]
  We start by fixing an eigenvalue $\mu_{j}^{2}$ on the cross-section
  and setting
  \begin{equation*}
    s = \sqrt{\left( \frac{n-1}{2}\right)^{2}  + \mu_{j}^{2}} .
  \end{equation*}
  The only possible resonances for this eigenvalue occur when
  \begin{equation*}
    b = \frac{1}{2} - \imath \lambda + s = -k \in \ZZ_{-},
  \end{equation*}
  i.e., for
  \begin{equation*}
    \lambda = -\imath \left( \frac{1}{2} + k + s\right).
  \end{equation*}

  If $s \in \frac{1}{2} + \NN$ and $\lambda$ as above, then we must have
  \begin{equation*}
    \imath\lambda = s + \frac{1}{2} - b  \in \NN,
  \end{equation*}
  and so $c \in \ZZ_{-}$ but $a\notin \ZZ_{-}$.  This is a case above
  in which the resolvent has a removable pole and so there is no
  resonance arising from this $\mu_{j}^{2}$.  

  Suppose now that $s \notin \frac{1}{2} + \NN$ and
  $\lambda$ is as above.  Note that we must then have
  \begin{equation*}
    \imath \lambda = s + \frac{1}{2} - b \notin \NN,
  \end{equation*}
  and so $c\in \ZZ_{-}$ if and only if $a\in \ZZ_{-}$.  In both of
  these cases we end up with a resonance.
\end{proof}

\section{Examples}
\label{sec:examples}

\subsection{Resonances for $\mathbb{H}^{n+1}$}
\label{sec:hyperbolic-ex}

In this section we recover the calculation of the location of
resonances on hyperbolic space.  In this case we have that the
cross-section $Y$ is $\mathbb{S}^{n}$ with its standard round metric.
The associated eigenvalues are given by 
\begin{equation*}
  \mu_{j}^{2} = j (j + n-1),
\end{equation*}
with multiplicities
\begin{equation*}
  m_{j} = \binom{n+j-2}{n} \frac{2j+n-1}{j},
\end{equation*}
where $\binom{n+j-2}{n}$ is the binomial coefficient.

We then have that
\begin{equation*}
  \sqrt{\left( \frac{n-1}{2}\right)^{2} + \mu_{j}^{2} } = j +
  \frac{n-1}{2}.
\end{equation*}
If $n$ is even, then $\frac{n-1}{2}$ is a half integer and the conclusion of Theorem~\ref{thm:main-thm} tells us that
odd-dimensional hyperbolic spaces have no resonances.  On the other
hand, if $n$ is odd, then $\frac{n-1}{2}$ is an integer so that 
even-dimensional hyperbolic spaces have resonances precisely at
\begin{equation*}
  \lambda_{j} = -\imath \left( \frac{1}{2} + \frac{n-1}{2} + j\right).
\end{equation*}
This recovers the well-known calculation of the resonances of
hyperbolic space.

\subsection{Resonances in the presence of a conic singularity}

Let us now take $n = 1$ so that $X$ is a surface with (potentially) an
isolated conic singularity.  This means that 
\begin{align*}
a & = \frac12 - i \lambda, \\
b  &=  \frac12 - i \lambda + \mu_j.
\end{align*}

Let $\hcone$ denote the hyperbolic cone of radius $\rho>0$, defined as
the product manifold $\hcone=\reals_+ \times\left(\mathbb{R} \big/
  2\pi\rho \mathbb{Z}\right)$, equipped with the metric $\g(r,\theta)
= \upd r^2 + \sinh^2 r \, \upd \theta^2$.  This is an incomplete
manifold which is locally isometric to $\HH^2$ away from the conic
singularity and hence hyperbolic there.  (In the case of the Euclidean cone ($\g(r,\theta) = \upd r^2
+r^2 \, \upd \theta^2$), see, for example, works of the second author
and collaborators~\cite{BFHM,blair2011strichartz}.)  Recall from above, that our methods
for computing resonances suggest that we should see resonances at
\[  \lambda = -i \left( \mu_j + \frac12  \right) - i k, \ \ k \in \NN.\]
In this case, the spectrum of $-\lap_{\rho}$ is easily described and we have
\[
\mu_j^2 \in \sigma ( -\Delta_\rho ) = \{ 0, \frac{1}{ \rho^2},
\frac{4}{ \rho^2}, \frac{9}{ \rho^2}, \dots \},
\]
the spectrum of the Laplacian on a circle with circumference $2 \pi
\rho$.  Thus, we observe that in such a case we have resonances for
\[
\lambda = -i \left( \frac12 + \frac{j}{\rho} - k \right)
\]
for $j \in \NN$ and $k = 1, 2, 3, \dots .$ In particular, we observe
that for cone angles much larger than $2 \pi$ ($\rho \gg1$), there are
resonances much closer to $\frac12$ than in the setting without a conic singularity, whereas for small cone angles $\rho < 1$, the resonances
introduced here occur at much larger values.

\def\cprime{$'$} \def\cftil#1{\ifmmode\setbox7\hbox{$\accent"5E#1$}\else
  \setbox7\hbox{\accent"5E#1}\penalty 10000\relax\fi\raise 1\ht7
  \hbox{\lower1.15ex\hbox to 1\wd7{\hss\accent"7E\hss}}\penalty 10000
  \hskip-1\wd7\penalty 10000\box7}

\end{document}